\documentclass[12pt]{article}

\usepackage{amssymb}
\usepackage{amsmath,amsthm}
\usepackage[dvips]{graphicx}
\usepackage{wrapfig}
\usepackage{bm}
\usepackage{url}

\newtheorem{theorem}{Theorem}[section]
\newtheorem{corollary}[theorem]{Corollary}
\newtheorem{lemma}[theorem]{Lemma}

\theoremstyle{definition}

\newtheorem{example}[theorem]{Example}
\theoremstyle{remark}

\title{On the warping sum of knots}

\author{Slavik Jablan\footnote{1952--2015}, 
Ayaka Shimizu \thanks{Department of Mathematics, National Institute of Technology, Gunma College, 580 Toriba-cho, Maebashi-shi, Gunma 371-8530, Japan. Email: shimizu@nat.gunma-ct.ac.jp }}
\date{\today}

\begin{document}

\maketitle

\begin{abstract}
The warping sum $e(K)$ of a knot $K$ is the minimal value of the sum of the warping degrees of a minimal diagram of $K$ with both orientations. 
In this paper, knots $K$ with $e(K) \le 3$ are characterized, and some knots $K$ with $e(K)=4$ are given. 
\end{abstract}

\section{Introduction}

In this paper, knot diagrams are oriented and on $S^2$, and they are considered up to ambient isotopy of $S^2$. 
An oriented knot diagram $D$ is {\it monotone} if one can travel along $D$ from a point on $D$ so that one meets each crossing as an over-crossing first. 
The {\it warping degree} of a knot diagram $D$, denoted by $d(D)$, is the smallest number of crossing changes required to obtain a monotone diagram from $D$ (defined by Kawauchi in \cite{AK}). 
The warping degree of an oriented knot diagram is dependent on both the diagram and the orientation, so it is not a knot invariant. 
On the other hand, this degree can be used to study a knot's orientation, its alternating behavior, its crossing number, and other properties (\cite{KS-ori, AS-k, AS-m}). 
The warping degree has been also defined and studied for links and spatial graphs (\cite{AK, AK-sg}), nanowords (\cite{TF}), virtual knot diagrams (\cite{OSY}) and so on. 
Similar concepts have been studied from various view points (see, for example, \cite{JH, LM}). 
In particular, the warping degree relates to a knot invariant called the {\it ascending number} of a knot (\cite{MO}, see also \cite{fujimura, fung, SJ, okuda}); 
the minimal warping degree $d(D)$ for all diagrams $D$ with an orientation of an unoriented knot $K$ is the ascending number $a(K)$ of $K$. 

For a knot diagram $D$, let $-D$ denote the diagram $D$ with its orientation reversed. 
The {\it warping sum} $e(D)$ of a knot diagram $D$ is the sum $d(D)+d(-D)$ of warping degrees of $D$ and $-D$. 
By definition, we have $e(-D)=e(D)$ and hence $e(D)$ is not orientation-dependent. 
We remark that $e(D)$ relates to the {\it span}, $\mathrm{spn}(D)$, of  the ``warping polynomial'' of a knot diagram $D$. 
We have the equality $\mathrm{spn}(D)=c(D)-e(D)$ (\cite{AS-p}), where $c(D)$ is the crossing number of $D$. 
(For the knot invariant $\mathrm{spn}(K)$, see also \cite{AL, CM}.)
For an (oriented or unoriented) knot $K$, the {\it warping sum} of the knot $K$, denoted by $e(K)$, is defined to be the minimal value of $e(D)$ for all {\it minimal diagrams} $D$ of $K$. 
In \cite{AS-k}, the inequality $e(K) \le c(K)-1$ is proven, giving an upper bound on $e(K)$. 
Furthermore, it is shown that equality holds if and only if $K$ is a prime alternating knot. 
In this paper, we provide a lower bound for $e(K)$. 
In particular, we show that any knot $K$ which is neither the trivial knot, $3_1$ nor $4_1$ has $e(K)\ge 4$ even if $K$ is not prime alternating (Theorem \ref{e-thm}). 
From this theorem, we determine some knots $K$ with $e(K)=4$. 
The rest of the paper is organized as follows: 
In Section 2, we investigate the warping sum $e(K)$, and we determine which knots have $e=0,1,2$ or $3$ by considering another new invariant $md(K)$, called the minimal warping degree. 
Then we give some knots $K$ with $e(K)=4$. 
In Section 3, we generalize the warping sum $e(K)$ to define a related invariant $\hat{e}(K)$ by considering all diagrams of $K$, not only minimal diagrams and we determine which knots have $\hat{e}=0,1,2$ or $3$.

\section{The warping sum $e(K)$}

In this section, we study the warping sum $e(K)$ and we give some knots $K$ with $e(K)=4$. 
In \cite{AS-k}, the following inequality is shown: 
\phantom{x}
\begin{theorem}[\cite{AS-k}]
For a nontrivial knot $K$, we have 
$$e(K) \le c(K)-1,$$
where the equality holds if and only if $K$ is a prime alternating knot. 
\label{ec-thm}
\end{theorem}
\phantom{x}

\noindent Since any minimal diagram of a prime alternating knot is a reduced alternating diagram, and the equality $e(D)=c(D)+1$ holds for any nontrivial alternating diagram of a knot (\cite{AS-k}), any minimal diagram $D$ of a nontrivial prime alternating knot $K$ has $e(D)=c(K)-1$. 
However, different minimal diagrams of the same prime alternating knot can give different breakdown of the warping degree $d(D)$ and $d(-D)$ as shown in Examples \ref{ex-7-6} and \ref{ex-8-12}. 

\phantom{x}
\begin{example}
Two minimal diagrams of $7_6$, $D_1$ and $D_2$ depicted in Figure \ref{7-6} have $d(D_1, -D_1)=(3,3)$ and $d(D_2, -D_2)=(2,4)$, where we denote the pair $( d(D), d(-D) )$ by $d(D, -D)$. 
Hence $d(D)$ is not preserved by some flype moves although $e(D)$ is preserved by flypes on reduced alternating diagrams of a prime alternating knot; $e(D_1)=e(D_2)=6$. 
\begin{figure}[ht]
\begin{center}
\includegraphics[width=70mm]{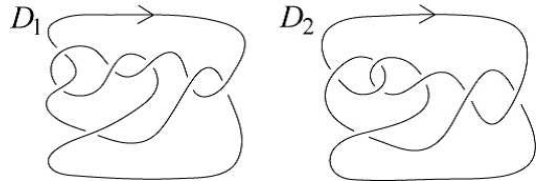}
\caption{Minimal diagrams $D_1$ and $D_2$ of $7_6$. }
\label{7-6}
\end{center}
\end{figure}
\label{ex-7-6}
\end{example}
\phantom{x} 

\begin{example}
Two minimal diagrams $D_1$ and $D_2$ of $K=8_{12}$ depicted in Figure \ref{8-12} have $d(D_1, -D_1)=(3,4)$ and $d(D_2, -D_2)=(2,5)$. 
\begin{figure}[ht]
\begin{center}
\includegraphics[width=70mm]{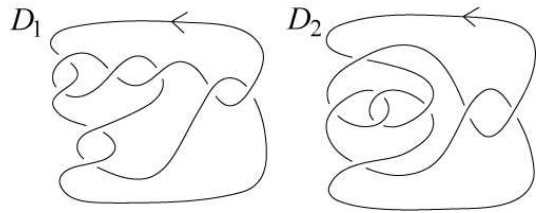}
\caption{Minimal diagrams $D_1$ and $D_2$ of $8_{12}$. }
\label{8-12}
\end{center}
\end{figure}
\label{ex-8-12}
\end{example}
\phantom{x}

\noindent We define the {\it minimal warping degree}, $md(K)$, of a knot $K$ to be the minimal value of the warping degree, $d(D)$, for all minimal diagrams $D$ of $K$ with all possible orientations. 
Note that the minimal warping degree $md(K)$ and the warping sum $e(K)$ are computable for prime knots $K$ with up to $12$ crossings by checking all the diagrams with up to $12$ crossings using {\it LinKnot} (\cite{JS}). 
By definition, we have the inequality $u(K) \le a(K) \le md(K)$ for the unknotting number $u(K)$ and the ascending number $a(K)$ of a knot $K$ since $a(K)$ is the minimal value of the warping degree for all diagrams of $K$, not only minimal diagrams. 
Knots with ascending number one are determined as follows: 

\phantom{x}
\begin{theorem}[Ozawa, \cite{MO}]
A knot $K$ has ascending number one if and only if $K$ is a twist knot. 
\label{ozawa-thm}
\end{theorem}
\phantom{x}

\noindent A {\it twist knot} is a knot whose Conway notation is ``$2 \ n$'' or ``$-2 \ -n$'' (where $n$ is a positive integer). 
In this paper, we consider twist knots to be only those twist knots with described by the positive integers ``$2 \ n$'' without loss of generality. 
We have the following: 

\phantom{x}
\begin{theorem}
A knot $K$ has minimal warping degree one if and only if $K=3_1$ or $4_1$. 
\label{31-41}
\end{theorem}
\phantom{x}

\noindent To prove Theorem \ref{31-41}, we prepare the following lemmas: 

\phantom{x}
\begin{lemma}
Each twist knot has a unique minimal diagram. 
\end{lemma}

\begin{proof}
It is known that links with Conway notation ``$p \ q$'' ($p$ and $q$ are positive integers) have only one minimal diagram on $S^2$ (see, for example, \cite{AS-rcc}). 
Hence any twist knot has only one minimal diagram. 
\end{proof}

\phantom{x}
\begin{lemma}
Let $n$ be a positive integer, and let $D$ be a knot diagram with Conway notation ``$2 \ n$''. 
We have $d(D, -D)=\left( \frac{n+1}{2}, \frac{n+1}{2} \right)$ if $n$ is odd, and 
$d(D, -D)= \left( \frac{n}{2}, \frac{n+2}{2} \right)$ or $\left( \frac{n+2}{2}, \frac{n}{2} \right)$ if $n$ is even. 
\end{lemma}

\begin{proof}
Since $D$ is alternating, we can obtain the warping degree with an orientation by starting at any point just before an over-crossing and counting the number of crossings that we pass under the first time we meet them as we travel around $D$ (\cite{AS-k}). 
For example, see Figure \ref{2n-twist}. 
\begin{figure}[ht]
\begin{center}
\includegraphics[width=40mm]{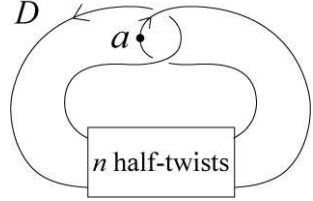}
\caption{For the case $n$ is odd and the orientation shown by the arrow, by taking the start point $a$ we meet $\frac{n+1}{2}$ crossing points as an under-crossing first in the $n$ half-twists. }
\label{2n-twist}
\end{center}
\end{figure}
\end{proof}
\phantom{x}

\noindent We have the following corollary: 

\phantom{x}
\begin{corollary}
Let $K$ be a twist knot with Conway notation ``$2 \ n$''. 
Then the minimal warping degree $md(K)$ equals $\lfloor \frac{n+1}{2} \rfloor .$
\label{cor-twist}
\end{corollary}
\phantom{x}

\noindent Note that the warping sum $e(K)$ equals $c(K)-1=n+1$ for twist knots $K$ with Conway notation ``$2 \ n$'' because $K$ is a prime alternating knot. 
Now we prove Theorem \ref{31-41}. \\

\phantom{x} 
\noindent {\it Proof of Theorem \ref{31-41}}. \ 
Let $K$ be a knot with $md(K)=1$. 
$K$ is not the trivial knot because $md(K)$ is not zero. 
Then we also have $a(K)=1$ from $a(K) \leq md(K)$. 
By Theorem \ref{ozawa-thm}, $K$ must be a twist knot. 
A twist knot $K$ with Conway notation ``$2 \ n$'' has $md(K)=1$ if and only if $n=1$ or $2$ by Corollary \ref{cor-twist}, that is, $K=3_1$ or $4_1$. 
\hfill$\square$

\phantom{x}

\noindent By Theorem \ref{31-41}, we can determine that the minimal warping degree $md(K)$ equals $2$ for some knots $K$. 
For example, we have $md(7_6)=md(8_{12})=2$ by Examples \ref{ex-7-6} and \ref{ex-8-12}. 
We show the following theorem:

\phantom{x}
\begin{theorem} 
Let $K$ be a knot. 
Then the small values for the warping sum $e(K)$ are determined as follows. \\
\noindent (0): $e(K)=0$ if and only if $K$ is the trivial knot. \\
\noindent (1): There are no knots $K$ with $e(K)=1$. \\
\noindent (2): $e(K)=2$ if and only if $K$ is the $3_1$ knot. \\
\noindent (3): $e(K)=3$ if and only if $K$ is the $4_1$ knot. \\
\label{e-thm}
\end{theorem}

\begin{proof}
(0): If a knot $K$ has a diagram $D$ with the warping degree $d(D)$ equals $0$, i.e., $K$ has a monotone diagram, then $K$ is the trivial knot. 
(1): If a knot diagram $D$ has the warping sum $e(D)=1$, then $d(D, -D)=(0,1)$ or $(1,0)$. 
This means $D$ is a diagram of the trivial knot, which has $e(K)=0$. 
(2): If $e(D)=2$, then $d(D, -D)=(0,2), (1,1)$ or $(2,0)$. 
For the case $(0,2)$ or $(2,0)$, $D$ represents the trivial knot, which has $e(K)=0$. 
For the case $(1,1)$, $D$ represents $3_1$ or $4_1$ if $D$ is a minimal diagram by Theorem \ref{31-41}. 
For the minimal diagram $D$ of $3_1$, we have $d(D, -D)=(1,1)$. 
For the minimal diagram $D$ of $4_1$, we have $d(D, -D)=(1,2)$ or $(2,1)$. 
Hence only $3_1$ has $e(K)=2$. 
(3): If $e(D)=3$, then $d(D, -D)=(0,3), (1,2), (2,1)$ or $(3,0)$. 
Similarly to (2), we can see that only $4_1$ has $e(K)=3$. 
\end{proof} 
\phantom{x}

\noindent We have the following corollary: 

\phantom{x}
\begin{corollary}
Let $K$ be a knot. 
If the value of the warping sum, $e(K)$, is $4$ or $5$, then the minimal warping degree, $md(K)$, equals $2$. 
\end{corollary}

\begin{proof}
Let $D$ be a minimal diagram of a knot $K$ with $e(D)=e(K)=4$ or $5$. 
Then $K$ must be neither the trivial knot, $3_1$ nor $4_1$ by Theorem \ref{e-thm}, and the warping degree of any minimal diagram of $K$ must be neither $0$ nor $1$ by Theorem \ref{31-41}. 
Hence $d(D, -D)$ is $(2,2)$, $(2,3)$ or $(3,2)$, and therefore $md(K)=2$. 
\end{proof}
\phantom{x}

\noindent Since a prime alternating knot $K$ has always the relation between the warping sum $e(K)$ and the crossing number $c(K)$ that $e(K)=c(K)-1$, we have $e(5_1)=e(5_2)=4$. 
In the following example, we give two knots which are non-alternating or non-prime with $e(K)=4$. 

\begin{example}
The non-alternating knot $8_{21}$ has $e(8_{21})=4$. 
The Granny knot $G$, a non-prime alternating knot, has $e(G)=4$. 
These values of the warping sum are realized by the minimal diagrams shown in Figure \ref{8-21}. 
\begin{figure}[ht]
\begin{center}
\includegraphics[width=70mm]{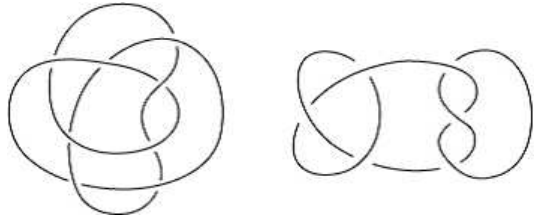}
\caption{The minimal diagrams of $8_{21}$ (left) and Granny knot (right) with $e=4$. }
\label{8-21}
\end{center}
\end{figure}
\end{example}

\section{The reduced warping sum $\hat{e}(K)$}

For a knot $K$, the warping sum $e(K)$ is defined to be the minimal value of the warping sum $e(D)$ for all minimal diagrams $D$ of $K$. 
By considering all diagrams $D$ including non-minimal diagrams, we might obtain a value of $e(D)$ (for a non-minimal diagram $D$ of $K$) that is smaller than $e(K)$. 
For example, the knot $6_3$ has $e(6_3)=5$, and it has a non-minimal diagram $D$ with $e(D)=4$ (see Figure \ref{6-3}). 
\begin{figure}[ht]
\begin{center}
\includegraphics[width=40mm]{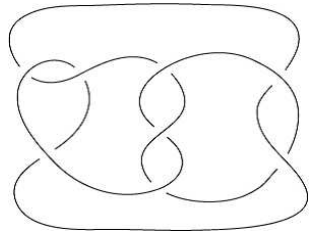}
\caption{The knot $6_3$ has a non-minimal diagram $D$ with $e(D)=4$. }
\label{6-3}
\end{center}
\end{figure}

\noindent We define the {\it reduced warping sum}, $\hat{e}(K)$, of a knot $K$ to be the minimum value of warping sum $e(D)$ over all possible diagrams $D$ of $K$. 
We show the following theorem: 

\phantom{x}
\begin{theorem}
Let $K$ be a knot. 
Then the small values for the reduced warping sum $\hat{e}(K)$ are determined as follows. \\
(0) $\hat{e}(K)=0$ if and only if $K$ is the trivial knot. \\
(1) There are no knots $K$ with $\hat{e}(K)=1$. \\
(2) $\hat{e}(K)=2$ if and only if $K$ is a twist knot. \\
(3) There are no knots $K$ with $\hat{e}(K)=3$.
\label{e-hat-thm}
\end{theorem}

\begin{proof}
(0) and (1): Similar to (0) and (1) of the proof of Theorem \ref{e-thm}. 
(2): If a diagram $D$ of a knot $K$ realizes $e(D)= \hat{e}(K)=2$, then $d(D, -D)$ should be $(1,1)$ since $K$ is not the trivial knot because $\hat{e}(K) \ne 0$. 
By Theorem \ref{ozawa-thm}, a non-trivial knot $K$ which has an oriented diagram $D$ with $d(D)=1$ is a twist knot. 
Using Ozawa's method in \cite{MO}, we can check that all twist knots have a diagram $D$ with $e(D)=2$ (see Figure \ref{twist-2}). 
\begin{figure}[ht]
\begin{center}
\includegraphics[width=110mm]{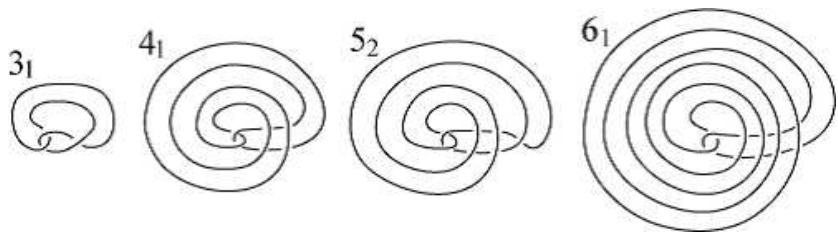}
\caption{Every twist knot has a diagram with warping degree equal to one, regardless of which orientation is chosen (Ozawa's method in \cite{MO}). }
\label{twist-2}
\end{center}
\end{figure}
(3): Since $3=0+3$ or $1+2$, a diagram $D$ with $e(D)=3$ represents the trivial knot or a twist knot. 
\end{proof}
\phantom{x}

\noindent By Theorem \ref{e-hat-thm}, we can conclude that $\hat{e}(6_3)=4$ (Figure \ref{6-3}).

\section*{Acknowledgment.} 
The second author thanks Allison Henrich for valuable discussion on e-mail and helping for editing the paper. 
She was partially supported by Grant for Basic Science Research Projects from The Sumitomo Foundation (160154).

\end{document}